\documentclass[12pt,reqno]{amsart}
\usepackage{etoolbox}
\makeatletter
\patchcmd\maketitle
  {\uppercasenonmath\shorttitle}
  {}
  {}{}
\patchcmd\maketitle
  {\@nx\MakeUppercase{\the\toks@}}
  {\the\toks@}
  {}
  {}{}
\patchcmd\@settitle{\uppercasenonmath\@title}{\Large}{}{}
\patchcmd\@setauthors
  {\MakeUppercase{\authors}}
  {\authors}
  {}{}
\makeatother
\usepackage{amsmath,amssymb,amsthm}
\usepackage{color}
\usepackage{url}
\usepackage{tikz-cd}
\usepackage[utf8]{inputenc}
\usepackage[T1]{fontenc}
\newtheorem{theorem}{Theorem}[section]

\newtheorem{corollary}{Corollary}[section]
\newtheorem{proposition}{Proposition}[section]
\newtheorem{lemma}{Lemma}[section]
\newtheorem{remark}{Remark}[section]

\numberwithin{equation}{section}

\usepackage[colorlinks=true,pagebackref=true]{hyperref} 

\begin{document}
\title[On several new results related to Richard's inequality]{On several new results related to Richard's inequality}
\author[Cristian Conde and Nicu\c sor Minculete] {\large{Cristian Conde and Nicu\c{s}or Minculete}}

\address{$^{[1]}$ Instituto de Ciencias, Universidad Nacional de General Sarmiento and Consejo Nacional de Investigaciones Cientıficas Tecnicas, Argentina.}
\email{\url{cconde@campus.ungs.edu.ar}}

\address{$^{[2]}$ Transilvania University of Bra\c{s}ov, Bra\c{s}ov, 500091, Rom{a}nia.}
\email{\url{minculete.nicusor@unitbv.ro}}



\date{\today}

\maketitle


\begin{abstract}
The main study of this article is the characterization of Richard's inequality, because it is closely related to Buzano's inequality. Finally, we present a new approach for Richard's inequality, where we use the Selberg operator.
\end{abstract}
\section{Introduction}
Lagrange shows the following identity:

\begin{equation}
\left( \overset{n}{\underset{i=1}{\sum }}a_{i}^{2}\right) \left( \overset{n}{%
\underset{i=1}{\sum }}b_{i}^{2}\right) =\left( \overset{n}{\underset{i=1}{%
\sum }}a_{i}b_{i}\right) ^{2}+\underset{1\leq i<j\leq n}{\sum }\left(
a_{i}b_{j}-a_{j}b_{i}\right) ^{2}.   \label{1}
\end{equation}%

A consequence of Lagrange's identity is the classical Cauchy-Schwarz inequality, in discrete case,
which states: if $\mathbf{a}=(a_{1},\ldots ,a_{n})$ and $\mathbf{b}%
=(b_{1},\ldots ,b_{n})$ are two $n-$tuples of real numbers, then

\begin{equation}
\left(a_{1}b_{1}+...+a_{n}b_{n}\right)^2\leq\left( a_{1}^{2}+...+a_{n}^{2}\right) \left(
b_{1}^{2}+...+b_{n}^{2}\right),   \label{2}
\end{equation}%
with equality holding if and only if $\mathbf{a}=\lambda \mathbf{b}$. This
result is called the \textit{Cauchy-Buniakowski-Schwarz inequality}. 

Several refinements of Cauchy-Buniakowski-Schwarz inequality can be found in
some papers (see \cite{2}, \cite{4}, \cite{5} and \cite{11}). We gave one of them: Ostrowski \cite{11}, showed the following:
if $\mathbf{x}=(x_{1},\ldots ,x_{n}),\mathbf{y}=(y_{1},\ldots ,y_{n})$ and $%
\mathbf{z}=(z_{1},\ldots ,z_{n})$ are $n-$tuples of real numbers such that $%
\mathbf{x}$ and $\mathbf{y}$ are not proportional and

\begin{gather}
\overset{n}{\underset{k=1}{\sum }}y_{k}z_{k}=0 \text{ and }\overset{n}{%
\underset{k=1}{\sum }}x_{k}z_{k}=1,\text{ then}  \label{3} \\
\overset{n}{\underset{k=1}{\sum }}y_{k}^{2}\ \diagup \overset{n}{\underset{%
k=1}{\sum }}z_{k}^{2}\leq \overset{n}{\underset{k=1}{\sum }}x_{k}^{2}\overset%
{n}{\underset{k=1}{\sum }}y_{k}^{2}-\left( \overset{n}{\underset{k=1}{\sum }}%
x_{k}y_{k}\right) ^{2}.  \notag
\end{gather}%
 In the framework of an inner product space $\mathcal{X}=\left( \mathcal{X},\left\langle \cdot
,\cdot \right\rangle \right) $ over the field of complex numbers $\mathbb{C}$
or real numbers $\mathbb{R}$, the Cauchy--Schwarz inequality (C-S), is given by the following:

\begin{equation}
\left\vert \left\langle x,y\right\rangle \right\vert \leq \left\Vert
x\right\Vert \cdot \left\Vert y\right\Vert   \label{4}
\end{equation}
for all $x,y\in \mathcal{X}$. The equality holds in \eqref{4} if and only if the vectors
$x$ and $y$ are linearly dependent, i.e., there exists a nonzero constant $\lambda\in\mathbb{C}$ so that $x=\lambda y$.

Buzano \cite{3} proved an extension of the Cauchy-Schwarz inequality, given
by the following:

\begin{equation}
\left\vert \left\langle a,x\right\rangle \left\langle x,b\right\rangle
\right\vert \leq \frac{1}{2}\left\Vert x\right\Vert ^{2}\left( \left\vert
\left\langle a,b\right\rangle \right\vert +\left\Vert a\right\Vert \cdot
\left\Vert b\right\Vert \right),  \label{5}
\end{equation}%
for any $x,a,b\in \mathcal{X}$. For $a=b$ in inequality  \eqref{5} we deduce the Cauchy-Schwarz inequality.

For a real inner space $\mathcal{X}$, Richard \cite{15},
gave the following inequality

\begin{equation}
\left\vert \left\langle a,x\right\rangle \left\langle x,b\right\rangle -%
\frac{1}{2}\left\Vert x\right\Vert ^{2}\left\langle a,b\right\rangle
\right\vert \leq \frac{1}{2}\left\Vert x\right\Vert ^{2}\left\Vert
a\right\Vert \cdot \left\Vert b\right\Vert,   \label{7}
\end{equation}%
for any $x,a,b\in \mathcal{X}$. It is trivial to see that the previous inequality  improves Buzano inequality. 

Precupanu \cite{14} mentioned an inequality related to the Richard inequality, for any $x,a,b\in \mathcal{X}$, we have

\begin{equation}
\frac{-\|a\|\|b\|+\langle a, b\rangle}{2}\leq \frac{\langle a, w\rangle \langle w, b\rangle}{\|w\|^2}+\frac{\langle a, z\rangle \langle z, b\rangle}{\|z\|^2}-2\frac{\langle a, w\rangle \langle w,z\rangle \langle z, b\rangle}{\|w\|^2\|z\|^2} \leq \frac{\|a\|\|b\|+\langle a, b\rangle}{2}, \label{7.1}
\end{equation}
for $\mathcal{X}$ a real inner space.

In \cite{7}, Gavrea proved an extention of Buzano's
inequality in an inner product space and 
 Popa and Ra\c{s}a showed in \cite{13}, that, for any $x,a,b\in \mathcal{X}$, we have

\begin{equation}
\left\vert {\Re}\left( \left\langle a,x\right\rangle \left\langle
x,b\right\rangle -\frac{1}{2}\left\Vert x\right\Vert ^{2}\left\langle
a,b\right\rangle \right) \right\vert \leq \frac{1}{2}\left\Vert x\right\Vert
^{2}\sqrt{\left\Vert a\right\Vert ^{2}\cdot \left\Vert b\right\Vert
^{2}-\left( {\Im}\left\langle a,b\right\rangle \right) ^{2}}, 
\label{8}
\end{equation}%
where $z=\Re(z)+i\Im(z)\in\mathbb{C}$.

In \cite{9}, Lupu and Schwarz gave the following inequality:

\begin{equation}
\left\Vert a\right\Vert ^{2}\left\vert \left\langle b,x\right\rangle
\right\vert^2 +\left\Vert b\right\Vert ^{2}\left\vert \left\langle
x,a\right\rangle \right\vert^2 +\left\Vert x\right\Vert
^{2}\left\vert \left\langle a,b\right\rangle \right\vert^2 \leq \left\Vert a\right\Vert
^{2}\left\Vert b\right\Vert ^{2}\left\Vert x\right\Vert ^{2}+2\left\vert
\left\langle a,b\right\rangle \left\langle b,x\right\rangle \left\langle
x,a\right\rangle \right\vert ,  \label{9}
\end{equation}
for any vectors $x,a,b\in \mathcal{X}$. This inequality gives us another refinement of the (C-S) inequality
\begin{equation*}
0\leq\frac{1}{\|x\|^2}\left(\|a\||\langle b,x\rangle|-\|b\||\langle x,a\rangle|\right)^2\leq \left\Vert
a\right\Vert^2 \cdot \left\Vert b\right\Vert^2-\left\vert \left\langle a,b\right\rangle \right\vert^2,   
\end{equation*}
for any vectors $x,a,b\in \mathcal{X}$, $x\neq 0$.

These inequalities, mentioned above, are applied to the theory of Hilbert $\mathbb{C}^{\ast }-$%
modules over non-commutative $\mathbb{C}^{\ast }-$algebras, see Aldaz \cite%
{1}, Pe\v{c}ari\'{c} and Raji\'{c} \cite{12} and Dragomir \cite{6}, \cite{5}.

The main study of this article is is the characterization of Richard's inequality, because it is closely related to Buzano's inequality. 
In Section 2 we are looking some bounds of the expression $\alpha\langle a,x\rangle \langle x,b\rangle -\beta\|x\|^2\langle a,b\rangle$ which is used in the study of some important inequalities, such as those given by Buzano, Richard, Ostrowski, Dragomir, Khosravi, Drnov\v{s}ek and Moslehian. In Section 3, we present a new approach for Richard's inequality, where we use the Selberg operator. We also give a result
which is the corresponding complex version of Precupanu's inequality.

\section{Main results}\label{s2}
First, we look at the expression $\alpha\langle a,x\rangle \langle x,b\rangle -\beta\|x\|^2\langle a,b\rangle$ as the scalar product of two vectors and we give an important identity.
\begin{lemma}
In an inner product space $\mathcal{X}$ over the
field of complex numbers $\mathbb{C}$, we have
\begin{equation}
\|\alpha\langle a,x\rangle x -\beta\|x\|^2a\|^2=\|x\|^2\left(|\langle a,x\rangle|^2|\beta-\alpha|^2+|\beta|^2\left\Vert \|x\|a-\frac{\langle a,x\rangle}{\|x\|}x\right\Vert^{2}\right) \label{10}
\end{equation}
for all $a,x\in \mathcal{X}$, $x\neq 0$, and for every $\alpha,\beta \in \mathbb{C}$.
\end{lemma}
\begin{proof}
For $\beta=0$ we obtain equality in relation of the statement. Next, we consider $\beta\neq 0$.
In \cite{9.1} we found the following identity:
\begin{equation*}
\left\Vert x+\alpha y\right\Vert ^{2}
=\left\vert \alpha \left\Vert y\right\Vert +\frac{\left\langle x,y\right\rangle }{\left\Vert y\right\Vert }\right\vert ^{2}+\left\Vert x-\frac{\left\langle x,y\right\rangle }{\left\Vert y\right\Vert ^{2}}y\right\Vert ^{2}
\end{equation*}
for all $x,y\in \mathcal{X}, y\neq 0$, and for every $\alpha \in \mathbb{C}$.

If we replace $\alpha$, $x$ and $y$ by $-\dfrac{\alpha}{\beta}\dfrac{\langle a,x\rangle}{\|x\|^2}$ (because $x\neq 0$), $a$ and $x$ in above identity, then we obtain

\begin{align*}
\left\Vert a-\frac{\alpha}{\beta}\frac{\langle a,x\rangle}{\|x\|^2} x\right\Vert ^{2}
&=\left\vert -\frac{\alpha}{\beta}\frac{\langle a,x\rangle}{\|x\|^2} \left\Vert x\right\Vert +\frac{\left\langle a,x\right\rangle }{\left\Vert x\right\Vert }\right\vert ^{2}+\left\Vert a-\frac{\left\langle a,x\right\rangle }{\left\Vert x\right\Vert ^{2}}x\right\Vert ^{2}\\
&=\frac{\left|\langle a,x\right\rangle|^2|\beta-\alpha|^2}{|\beta|^2\left\Vert x\right\Vert^2 }+\left\Vert a-\frac{\left\langle a,x\right\rangle }{\left\Vert x\right\Vert ^{2}}x\right\Vert ^{2}.
\end{align*}

Therefore, multiplying by $|\beta|^2\|x\|^4$, in above relation, we deduce the identity of the statement.
\end{proof}

\begin{remark}
Taking into account that $|\beta|^2\left\Vert \|x\|a-\dfrac{\langle a,x\rangle}{\|x\|}x\right\Vert^{2}\geq 0$, then, from equality \eqref{10}, we find
\begin{equation}
\|\alpha\langle a,x\rangle x -\beta\|x\|^2a\|\geq\|x\||\langle a,x\rangle||\beta-\alpha|  \label{11}
\end{equation}
for all $a,x\in \mathcal{X}$, and for every $\alpha \in \mathbb{C}$. The case $x=0$ is checked separately.
Since, we have $\left\Vert \|x\|a-\frac{\langle a,x\rangle}{\|x\|}x\right\Vert^{2}=\|a\|^2\|x\|^2-|\langle a,x\rangle|^2$, equality \eqref{10} becomes
\begin{equation}
\|\alpha\langle a,x\rangle x -\beta\|x\|^2a\|^2=\|x\|^2\left(|\alpha-\beta|^2|\langle a,x\rangle|^2+|\beta|^2\|a\|^2\|x\|^2-|\beta|^2|\langle a,x\rangle|^2\right) \label{12}
\end{equation}
for all $a,x\in \mathcal{X}$, and for every $\alpha,\beta \in \mathbb{C}$, with separate verification for the case $x=0$.

For $\alpha=2$ and $\beta=1$ in identity \eqref{10} we obtain the following \cite{9.1}:
\begin{equation}
\left\|\langle a,x\rangle x -\frac{1}{2}\|x\|^2a\right\|=\frac{1}{2}\|x\|^2\|a\|  \label{13}
\end{equation}
for all $a,x\in \mathcal{X}$.

We known the algebraic inequality $pp_1+qq_1\leq\max\{p,q\}(p_1+q_1)$ for all $p,p_1,q,q_1\geq 0$. If we take $p=|\alpha-\beta|^2$, $q=|\beta|^2$, $p_1=|\langle a,x\rangle|^2$ and $q_1=\|a\|^2\|x\|^2-|\langle a,x\rangle|^2$, then we have
\begin{align*}
&\|\alpha\langle a,x\rangle x -\beta\|x\|^2a\|^2\leq\|x\|^2\max\{|\alpha-\beta|^2,|\beta|^2\}\left(|\langle a,x\rangle|^2+\|a\|^2\|x\|^2-|\langle a,x\rangle|^2\right)\\
&=\max\{|\alpha-\beta|^2,|\beta|^2\}\|a\|^2\|x\|^4,
\end{align*}
which is equivalent to
\begin{equation}
\|\alpha\langle a,x\rangle x -\beta\|x\|^2a\|\leq\max\{|\alpha-\beta|,|\beta|\}\|a\|\|x\|^2  \label{14}
\end{equation}
for all $a,x\in \mathcal{X}$ and for every $\alpha,\beta \in \mathbb{C}$.

Also here, it should be mentioned that, combining inequalities \eqref{11} and \eqref{14} we obtain an improvement of the Cauchy--Schwarz inequality, thus
\begin{equation}
|\langle a,x\rangle|\leq\frac{\|\alpha\langle a,x\rangle x -|\beta|\|x\|^2a\|}{|\alpha-\beta|\|x\|}\leq\|a\|\|x\|  \label{15}
\end{equation}
for all $a,x\in \mathcal{X}$, $x\neq 0$ and for every $\alpha,\beta \in \mathbb{C}$ with $|\alpha-\beta|\geq |\beta|>0$.
\end{remark}

\begin{theorem}
In an inner product space $\mathcal{X}$ over the
field of complex numbers $\mathbb{C}$, we have
\begin{equation}
|\alpha\langle a,x\rangle \langle x,b\rangle -\beta\|x\|^2\langle a,b\rangle|\leq\max\{|\beta|,|\alpha-\beta|\}\|x\|^2\|a\|\|b\|  \label{16}
\end{equation}
for all $a,b,x\in \mathcal{X}$ and for every $\alpha,\beta \in \mathbb{C}$.
\end{theorem}
\begin{proof}
For $\beta=0$ the inequality of the statement are true. For $\beta\neq 0$, using the Cauchy--Schwarz inequality and inequality \eqref{14}, we deduce the following:
\begin{align*}
&|\alpha\langle a,x\rangle \langle x,b\rangle -\beta\|x\|^2\langle a,b\rangle|=|\langle\alpha\langle a,x\rangle x -\beta\| x\|^{2} a,b\rangle|\\
&\overset{(C-S)}{\leq} \|\langle\alpha\langle a,x\rangle x -\beta\| x\|^{2} a\|\|b\|\\
&\overset{\textcolor{red}{\eqref{14}}}{\leq}\max\{|\beta|,|\alpha-\beta|\}\|x\|^2\|a\|\|b\|.
\end{align*}
Therefore, the inequality of the statement was proven.
\end{proof}

\begin{remark} 
We take $\beta=1$ in inequality \eqref{16}, thus we show that
\begin{equation}
|\alpha\langle a,x\rangle \langle x,b\rangle -\|x\|^2\langle a,b\rangle|\leq\max\{1,|\alpha-1|\}\|x\|^2\|a\|\|b\|  \label{17}
\end{equation}
for all $a,b,x\in \mathcal{X}$ and for every $\alpha \in \mathbb{C}$. 
This inequality is given by Khosravi, Drnov\v{s}ek and Moslehian \cite{8} as an extension of Buzano's inequality given as a particularization in the context of Hilbert $C^*$--modules. We mentioned the fact that this inequality was studied by Dragomir in \cite{5}, when $|\alpha-1|=1$.

For $\alpha=2$ in relation \eqref{17}, we obtain Richard's inequality. 
\end{remark}

\begin{theorem}
In an inner product space $\mathcal{X}$ over the
field of complex numbers $\mathbb{C}$, we have
\begin{equation}
0\leq\frac{\left\Vert x\right\Vert ^{2}}{\left\Vert b\right\Vert ^{2}}\left| 
\alpha\frac{\left\langle a,x\right\rangle \left\langle x,b\right\rangle }{\left\Vert x\right\Vert ^{2}}-\beta\left\langle a,b\right\rangle \right| ^{2}\leq |\alpha-\beta|^2|\langle a,x\rangle|^2+|\beta|^2\left(
\left\Vert a\right\Vert ^{2}\left\Vert x\right\Vert ^{2}-\left|\left\langle
a,x\right\rangle \right| ^{2}\right)  \label{18}
\end{equation}
for all $a,b,x\in \mathcal{X}$, $b,x\neq 0$, and for every $\alpha,\beta \in \mathbb{C}$.
\end{theorem}
\begin{proof}
For $x\neq 0$ and $b\neq 0$, we make the following calculations:
\begin{align*}
&|\alpha\langle a,x\rangle\langle x,b\rangle -\beta\| x\|^{2}\langle a,b\rangle|^2\\
&=|\langle\alpha\langle a,x\rangle x -\beta\| x\|^{2} a,b\rangle|^2\\
&\overset{(C-S)}{\leq}\|\alpha\langle a,x\rangle x -\beta\| x\|^{2} a\|^2\|b\|^2\\
&\overset{\eqref{12}}{=}\|x\|^2\|b\|^2\left(|\alpha-\beta|^2|\langle a,x\rangle|^2+|\beta|^2\left(\|a\|^2\|x\|^2-|\langle a,x\rangle|^2\right)\right).
\end{align*}
It follows that
\begin{align*}
&\|x\|^4\left| 
\alpha\frac{\left\langle a,x\right\rangle \left\langle x,b\right\rangle }{\left\Vert x\right\Vert ^{2}}-\beta\left\langle a,b\right\rangle \right| ^{2}\\
&\leq \|x\|^2\|b\|^2\left(|\alpha-\beta|^2|\langle a,x\rangle|^2+|\beta|^2\left(\|a\|^2\|x\|^2-|\langle a,x\rangle|^2\right)\right).
\end{align*}
But, since $b,x\neq 0$, we dividing by $\|x\|^2\|b\|$ and we deduce the inequality of the statement.
\end{proof}

\begin{corollary}
In an inner product space $\mathcal{X}$ over
the field of complex numbers $\mathbb{C}$, the following inequality
\begin{equation}
0\leq\frac{\left\Vert x\right\Vert ^{2}}{\left\Vert b\right\Vert ^{2}}\left| 
\frac{\left\langle a,x\right\rangle \left\langle x,b\right\rangle }{%
\left\Vert x\right\Vert ^{2}}-\left\langle a,b\right\rangle \right| ^{2}\leq
\left\Vert a\right\Vert ^{2}\left\Vert x\right\Vert ^{2}-|\langle
a,x\rangle| ^{2}   \label{19}
\end{equation}
holds, for all $a,b,x\in \mathcal{X},x\neq 0,$ and $b\neq 0.$
\end{corollary}
\begin{proof}
In relation \eqref{18}, we take $\alpha=\beta\neq 0$ and we deduce the inequality of the statement.
\end{proof}

\begin{remark}
	If we take $%
	\left\langle x,b\right\rangle =0$, in inequality $\left( \ref{19}\right) $,
	then we obtain 
	\begin{equation}\label{GenOst}
	\frac{\left\Vert x\right\Vert ^{2}}{\left\Vert b\right\Vert ^{2}}\left| 
	\left\langle a,b\right\rangle \right| ^{2}\leq
	\left\Vert a\right\Vert ^{2}\left\Vert x\right\Vert ^{2}-|\langle
	a,x\rangle| ^{2} 
	\end{equation}%
	for all $a,b,x\in \mathcal{X},b\neq 0,x\neq 0.$ This inequality was obtained by Dragomir and Go\c sa in \cite{DG}.

In addition, if we consider $\left\langle a,b\right\rangle =1$ (or $|\left\langle a,b\right\rangle| =1$) in inequality \eqref{GenOst}, then
we find the inequality of Ostrowski for inner product spaces over the
field of complex numbers,

\begin{equation}
\frac{\left\Vert x\right\Vert ^{2}}{\left\Vert b\right\Vert ^{2}}\leq
\left\Vert a\right\Vert ^{2}\left\Vert x\right\Vert ^{2}-|\langle
a,x\rangle| ^{2}    \label{20}
\end{equation}%
for all $a,b,x\in \mathcal{X},b\neq 0,x\neq 0.$

The inequality of Ostrowski for inner product spaces over the field of real numbers was studied in \cite{9.1}. It is easy to see that for $a,b,x\in \mathbb{R}^{n}\ $we obtain inequality $%
\left( \ref{3}\right) $.
\end{remark}
\begin{theorem}
In an inner product space $\mathcal{X}$ over the field real or complex numbers, for any vectors $x,a,b\in \mathcal{X}$, $a\neq 0$ and $\alpha,\beta\in\mathbb{C}$, $\alpha\neq\beta$, where $|\alpha-\beta|\leq |\beta|$ with $\beta\neq 0$, we have
\begin{equation}
|\beta|\|x\|^2\|a\|\|b\|-|\alpha\langle a,x\rangle\langle x,b\rangle -\beta\| x\|^{2}\langle a,b\rangle|\geq \frac{\left(|\beta|^2-|\alpha-\beta|^2\right)\|b\||\langle a,x\rangle|^2}{2|\beta|\|a\|}\geq
0.    \label{21}
\end{equation}
\end{theorem}
\begin{proof}
If we have $x=0$ or $b=0$, then the relation of the statement is true. For $x\neq 0$ and $b\neq 0$, we make the following calculations:
\begin{align*}
&|\beta|\|x\|^2\|a\|\|b\|-|\alpha\langle a,x\rangle\langle x,b\rangle -\beta\| x\|^{2}\langle a,b\rangle|\\
&=|\beta|\|x\|^2\|a\|\|b\|-|\langle\alpha\langle a,x\rangle x -\beta\| x\|^{2} a,b\rangle|\\
&=\frac{|\beta|^2\|x\|^4\|a\|^2\|b\|^2-|\langle\alpha\langle a,x\rangle x -\beta\| x\|^{2} a,b\rangle|^2}{|\beta|\|x\|^2\|a\|\|b\|+|\langle\alpha\langle a,x\rangle x -\beta\| x\|^{2} a,b\rangle|}\\
&\overset{(C-S)}{\geq}\frac{|\beta|^2\|x\|^4\|a\|^2\|b\|^2-\|\langle\alpha\langle a,x\rangle x -\beta\| x\|^{2} a\|^2\|b\|^2}{|\beta|\|x\|^2\|a\|\|b\|+|\langle\alpha\langle a,x\rangle x -\beta\| x\|^{2} a,b\rangle|}\\
&\overset{\eqref{12}}{=}\frac{|\beta|^2\|x\|^4\|a\|^2\|b\|^2-|\alpha-\beta|^2\|x\|^2|\langle a,x\rangle|^2\|b\|^2-|\beta|^2\|a\|^2\|b\|^2\|x\|^4+|\beta|^2\|x\|^2\|b\|^2|\langle a,x\rangle|^2}{|\beta|\|x\|^2\|a\|\|b\|+|\langle\alpha\langle a,x\rangle x -\beta\| x\|^{2} a,b\rangle|}\\
&=\frac{\left(|\beta|^2-|\alpha-\beta|^2\right)\|x\|^2\|b\|^2|\langle a,x\rangle|^2}{|\beta|\|x\|^2\|a\|\|b\|+|\langle\alpha\langle a,x\rangle x -\beta\| x\|^{2} a,b\rangle|}\\
&\overset{\eqref{14}}{\geq}\frac{\left(|\beta|^2-|\alpha-\beta|^2\right)\|x\|^2\|b\|^2|\langle a,x\rangle|^2}{|\beta|\|x\|^2\|a\|\|b\|+\max\{|\beta|,|\alpha-\beta|\}\|x\|^2|\|a\|\|b\|}=\frac{\left(|\beta|^2-|\alpha-\beta|^2\right)\|x\|^2\|b\|^2|\langle a,x\rangle|^2}{2|\beta|\|x\|^2\|a\|\|b\|}\\
&=\frac{\left(|\beta|-|\alpha-\beta|^2\right)\|b\||\langle a,x\rangle|^2}{2|\beta|\|a\|}.
\end{align*}
Consequently, the inequality of the statement is true.
\end{proof}

\begin{theorem}
In an inner product space $\mathcal{X}$ over the field real or complex numbers, for any nonzero vectors $x,a,b\in \mathcal{X}$ and $\alpha,\beta \in \mathbb{C}$, $\max\{|\alpha-\beta|,|\beta|\}\neq 0$, we have

\begin{align}
&\max\{|\alpha-\beta|,|\beta|\}\left\Vert x\right\Vert ^{2}\left\Vert a\right\Vert \cdot
\left\Vert b\right\Vert -\left\vert\alpha \left\langle a,x\right\rangle
\left\langle x,b\right\rangle -\beta\left\Vert x\right\Vert
^{2}\left\langle a,b\right\rangle \right\vert \nonumber\\
&\geq \frac{A(\alpha,\beta)}{2\max\{|\alpha-\beta|,|\beta|\}\left\Vert x\right\Vert ^{2}\left\Vert a\right\Vert \cdot
\left\Vert b\right\Vert}\geq 0,   \label{22}
\end{align}
where
\begin{equation*}
A(\alpha,\beta)=\left(|\alpha| \left\vert \left\langle a,x\right\rangle \right\vert \left(
\left\Vert x\right\Vert ^{2}\left\Vert b\right\Vert ^{2}-\left\vert
\left\langle x,b\right\rangle \right\vert ^{2}\right)^{\frac{1}{2}} -|\beta|\left\Vert
x\right\Vert ^{2}\left( \left\Vert a\right\Vert ^{2}\left\Vert b\right\Vert
^{2}-\left\vert \left\langle a,b\right\rangle \right\vert ^{2}\right)^{\frac{1}{2}}
\right) ^{2}.
\end{equation*}
\end{theorem}
\begin{proof}
For all $x,y\in \mathcal{X},$ and $y\neq 0$, we have the following equality: 
\begin{equation*}
\left\Vert \left\Vert y\right\Vert x-\frac{\left\langle x,y\right\rangle }{%
\left\Vert y\right\Vert }y\right\Vert ^{2}=\left\Vert x\right\Vert
^{2}\left\Vert y\right\Vert ^{2}-\left\vert \left\langle x,y\right\rangle
\right\vert ^{2}.
\end{equation*}
Hence, using the Cauchy-Schwarz inequality to the denominator, we have the relation

\begin{equation*}
\left\Vert x\right\Vert \cdot \left\Vert y\right\Vert -\left\vert
\left\langle x,y\right\rangle \right\vert =\frac{\left\Vert \left\Vert
y\right\Vert x-\frac{\left\langle x,y\right\rangle }{\left\Vert y\right\Vert 
}y\right\Vert ^{2}}{\left\Vert x\right\Vert \cdot \left\Vert y\right\Vert
+\left\vert \left\langle x,y\right\rangle \right\vert }\geq \frac{\left\Vert
\left\Vert y\right\Vert x-\frac{\left\langle x,y\right\rangle }{\left\Vert
y\right\Vert }y\right\Vert ^{2}}{2\left\Vert x\right\Vert \cdot \left\Vert
y\right\Vert }.
\end{equation*}%
In this inequality, we replace $x$ and $y$ by $\alpha\left\langle
a,x\right\rangle x-\beta\left\Vert x\right\Vert ^{2}a$ and $b$ in the above inequality and using the inequality \eqref{14}, i.e. $\|\alpha\langle a,x\rangle x -\beta\|x\|^2a\|\leq\max\{|\alpha-\beta|,|\beta|\}\|a\|\|x\|^2$, implies

\begin{align*}
&\left\Vert \alpha\left\langle
a,x\right\rangle x-\beta\left\Vert x\right\Vert ^{2}a\right\Vert
\left\Vert b\right\Vert -\left\vert \alpha\left\langle a,x\right\rangle
\left\langle x,b\right\rangle -\beta\left\Vert x\right\Vert
^{2}\left\langle a,b\right\rangle \right\vert \\
&\geq \frac{\left\Vert
u-v\right\Vert ^{2}}{2\left\Vert \alpha\left\langle
a,x\right\rangle x-\beta\left\Vert x\right\Vert ^{2}a\right\Vert \left\Vert b\right\Vert }\geq \frac{\left\Vert
u-v\right\Vert ^{2}}{2\max\{|\alpha-\beta|,|\beta|\}\|a\|\|x\|^2},
\end{align*}%
where $u=\alpha\left\langle a,x\right\rangle \left( \left\Vert b\right\Vert x-%
\dfrac{\left\langle x,b\right\rangle }{%
\left\Vert b\right\Vert }b\right) $ and $v=\beta\left\Vert
x\right\Vert ^{2}\left( \left\Vert b\right\Vert a-\dfrac{\left\langle
a,b\right\rangle }{\left\Vert b\right\Vert }b\right) $.

But, we have that $\left\Vert u\right\Vert =|\alpha|\left\vert \left\langle
a,x\right\rangle \right\vert \left( \left\Vert x\right\Vert ^{2}\left\Vert
b\right\Vert ^{2}-\left\vert \left\langle x,b\right\rangle \right\vert
^{2}\right)^{\frac{1}{2}} $ and

$\left\Vert v\right\Vert =|\beta|\left\Vert x\right\Vert ^{2}\left(
\left\Vert a\right\Vert ^{2}\left\Vert b\right\Vert ^{2}-\left\vert
\left\langle a,b\right\rangle \right\vert ^{2}\right)^{\frac{1}{2}}.$

Therefore, since $\left\Vert u-v\right\Vert ^{2}\geq \left( \left\Vert u\right\Vert
-\left\Vert v\right\Vert \right)^{2}$, $u,v\in X$, we obtain the
inequality of the statement.
\end{proof}

\begin{remark}  For $\alpha=1$ and $\beta=\frac{1}{2}$ we obtain an important inequality given in \cite{9.1}, thus
\begin{align}
&\frac{1}{2}\left\Vert x\right\Vert ^{2}\left\Vert a\right\Vert \cdot
\left\Vert b\right\Vert -\left\vert\left\langle a,x\right\rangle
\left\langle x,b\right\rangle -\frac{1}{2}\left\Vert x\right\Vert
^{2}\left\langle a,b\right\rangle \right\vert \nonumber\\
&\geq \frac{A}{\left\Vert x\right\Vert ^{2}\left\Vert a\right\Vert \cdot
\left\Vert b\right\Vert}\geq 0,  \label{23}
\end{align}
where
\begin{equation*}
A=A\left(1,\frac{1}{2}\right)=\left(\left\vert \left\langle a,x\right\rangle \right\vert \left(
\left\Vert x\right\Vert ^{2}\left\Vert b\right\Vert ^{2}-\left\vert
\left\langle x,b\right\rangle \right\vert ^{2}\right)^{\frac{1}{2}} -\frac{1}{2}\left\Vert
x\right\Vert ^{2}\left( \left\Vert a\right\Vert ^{2}\left\Vert b\right\Vert
^{2}-\left\vert \left\langle a,b\right\rangle \right\vert ^{2}\right)^{\frac{1}{2}}
\right) ^{2}.
\end{equation*}
This inequality represents an improvement of Richard's inequality, given thus:
\begin{equation*}
\left\vert \left\langle a,x\right\rangle \left\langle x,b\right\rangle -%
\frac{1}{2}\left\Vert x\right\Vert ^{2}\left\langle a,b\right\rangle
\right\vert \leq \frac{1}{2}\left\Vert x\right\Vert ^{2}\left\Vert
a\right\Vert \cdot \left\Vert b\right\Vert -\frac{A}{\left\Vert x\right\Vert
^{2}\left\Vert a\right\Vert \cdot \left\Vert b\right\Vert }.
\end{equation*}
\end{remark}

\begin{corollary} 
In an inner product space $\mathcal{X}$ over the field real or complex numbers, for any nonzero vectors $x,a,b\in \mathcal{X}$ and $\alpha,\beta \in \mathbb{C}$, $\beta\neq 0$, we have
		\begin{align}
		&\left\Vert x\right\Vert
		^{2}\left(|\beta||\langle a,b\rangle|-\max\{|\alpha-\beta|,|\beta|\}\left\Vert a\right\Vert \cdot
		\left\Vert b\right\Vert\right)+\frac{A(\alpha,\beta)}{2\max\{|\alpha-\beta|,|\beta|\}\left\Vert x\right\Vert ^{2}\left\Vert a\right\Vert \cdot
			\left\Vert b\right\Vert}\nonumber\\
		&\leq|\alpha||\langle a,x\rangle\left\langle x,b\right\rangle|  \nonumber\\
		&\leq \left\Vert x\right\Vert
		^{2}\left(|\beta||\langle a,b\rangle|+\max\{|\alpha-\beta|,|\beta|\}\left\Vert a\right\Vert \cdot
		\left\Vert b\right\Vert\right)-\frac{A(\alpha,\beta)}{2\max\{|\alpha-\beta|,|\beta|\}\left\Vert x\right\Vert ^{2}\left\Vert a\right\Vert \cdot
			\left\Vert b\right\Vert}.    \label{24}
		\end{align}
		\end{corollary}
	\begin{proof}
		By using inequality \eqref{22}, and from the continuity property of the
		modulus, i.e., $\left\vert u -v \right\vert \geq \left\vert
		\left\vert u \right\vert -\left\vert v \right\vert \right\vert $, $u,v \in \mathbb{C},$ we easily  deduce the desired inequality.
		\end{proof}

We obtain from  inequality \eqref{24}  a refinement of Buzano's inequality, as follows.
\begin{proposition}
In an inner product space $\mathcal{X}$ over the field real or complex numbers, for any nonzero vectors $x,a,b\in \mathcal{X}$, we have
			\begin{align}
		&|\langle a,x\rangle\left\langle x,b\right\rangle|  \nonumber\\
		&\leq \left\Vert x\right\Vert
		^{2}\left(\frac 12|\langle a,b\rangle|+\frac 12\left\Vert a\right\Vert \cdot
		\left\Vert b\right\Vert\right)-\frac{1}{\left\Vert x\right\Vert ^{2}\left\Vert a\right\Vert \cdot
			\left\Vert b\right\Vert}\max\left\{A\left(1,\frac{1}{2}\right), \frac14 A\left(2, 1\right)\right\}.    \label{25}
		\end{align}
			\end{proposition} \begin{proof}
			The refinement of Buzano's inequality follows from inequality \eqref{22}, when $\alpha=1$ and $\beta=\frac{1}{2}$ or $\alpha=2$ and $\beta=1$. 
		\end{proof} 

\section{A new approach for Richard's inequality}
Throughout this section, we denote by $\mathcal{X}$ a complex  Hilbert
space, i.e. an complete and inner product space, where the inner product $\langle\cdot,\cdot\rangle$ and the corresponding norm $\|\cdot\|$ are defined.
 We denote the $C^*$-algebra of all bounded linear operators acting on $\mathcal{X}$ as $\mathcal{B}(\mathcal{X})$ and the identity operator is represented by $I$. The operator norm of $T$ is given by 
 \begin{equation*}
 \|T\|=\sup\{\|Tx\|: \|x\|=1, x\in \mathcal{H}\}=\sup\{|\langle Tx,y\rangle|: \|x\|=\|y\|=1, x,y\in \mathcal{H}\}.
 \end{equation*}

For an operator $T\in \mathcal{B}(\mathcal{\mathcal{X}})$, the nullspace of $T$ is denoted as $\mathcal{N}(T)$, and $T^*$ represents its adjoint. We define a positive operator, denoted as $T\geq 0$, as an operator that satisfies $\langle Tx,x\rangle \geq 0$ for all $x\in \mathcal{H}$. Moreover, the order relation $T\geq S$ is introduced for self-adjoint operators, which holds when $T-S\geq 0$.

Given a bounded linear operator $T$ defined on $\mathcal{X}$, recall that he numerical radius, denoted as $\omega(T)$, is defined as the supremum (or maximum) of the absolute values of the numbers in the numerical range $W(T)$, more precisely
\begin{equation*}
\omega(T)=\sup\{|\lambda|: \lambda \in W(T)\},
\end{equation*}
where
$W(T)=\{\langle Tx, x\rangle : x\in \mathcal{H}, \|x\|=1\}.$

For the subsequent discussion, it is important to recall that the expression $x \otimes y$ represents a rank one operator defined by $x \otimes y(z) = \langle z, y\rangle x$, where $x$, $y$, and $z$ are vectors in the space $\mathcal{X}$.  So, we can rewrite inequality  \eqref{16}, as follows:
\begin{equation}
|\langle[\alpha(x \otimes x)-\beta \|x\|^2 I] a, b\rangle| \leq\max\{|\beta|,|\alpha-\beta|\}\|x\|^2\|a\|\|b\|,  \label{26}
\end{equation}
where $a, b, x\in \mathcal{X}$ and $\alpha, \beta\in \mathbb{C}.$

Taking the supremum in relation \eqref{26} for $\|a\|=\|b\|=1$, we deduce
\begin{equation*}
\|\alpha(x \otimes x)-\beta \|x\|^2 I\|\leq\max\{|\beta|,|\alpha-\beta|\}\|x\|^2.
\end{equation*}

It is well-known that if $x\in \mathcal{X}$ with $\|x\|=1$, then $P_x=x \otimes x$ is the orthogonal projection on ${\rm span}\{x\}.$

\begin{remark}
	If in inequality \eqref{26}, we assume that $\alpha=2, \beta=1$ and $x\in \mathcal{X}$ is a norm one vector, then 
	\begin{equation*}
\|2P_x-I\|=\|2(x \otimes x)- I\|\leq\max\{|1|,|2-1|\}=1.
	\end{equation*}
	Fuji and Kubo \cite{fujiikubo}, used this inequality to give a simpler proof of Buzano's inequality.
	\end{remark}


A significant inequality was discovered by A. Selberg (\cite[p. 394]{MPF}). If we consider vectors $x, z_1, \dots, z_n$ in $\mathcal{X}$, where $z_i \neq 0$ for all $i\in\{1,\dots,n\}$, we can consider Selberg's inequality, which asserts that:
\begin{equation}
\sum_{i=1}^{n}\frac{\left\vert \langle x, z_{i}\rangle \right\vert ^{2}}{%
	\sum_{j=1}^{n}\left\vert \langle z_{i},z_{j}\rangle \right\vert }\leq
\left\Vert x\right\Vert ^{2}. \label{Selberg}
\end{equation}\\
In \cite{ACDF}, was introduced the Selberg operator defined as follows:
	given a subset $\mathcal{Z}=\{z_i: i=1, \cdots, n\}$ of nonzero vectors in the space $\mathcal{X}$, the Selberg operator $S_{\mathcal{Z}}$ is defined by
	\begin{equation*}
	S_{\mathcal{Z}}= \sum_{i=1}^n\frac{z_i \otimes z_i}{ \sum_{j=1}^n |\langle z_i, z_j\rangle|}\in \mathcal{B}(\mathcal{X}).
	\end{equation*}\\
\begin{remark}
			Selberg's inequality gives us another refinement of the (C-S) inequality, since if $a, b\in \mathcal{X}$ with $a$ and $b$ nonzero vectors in $\mathcal{X}$, then 
			\begin{align*}
			0\leq \left(\|a\|^2-\langle S_{\{b\}}a, a\rangle\right)\left(\|b\|^2-\langle S_{\{a\}}b, b\rangle\right)\leq \|b\|^2\left(\|a\|^2-\langle S_{\{b\}}a, a\rangle\right),
			\end{align*}
			or equivalently, 
			\begin{align*}
			0\leq \left(\|a\|^2-\langle S_{\{b\}}a, a\rangle\right)\left(\|b\|^2-\langle S_{\{a\}}b, b\rangle\right)\leq \|b\|^2\|a\|^2-\left\vert \left\langle a,b\right\rangle \right\vert^2.
			\end{align*}
			\end{remark}
Now, we will express Richard's inequality using an appropiate Selberg operator, more precisely 
\begin{equation}\label{RichardSelberg}
\left| \langle S_{\mathcal{Z}} a, b\rangle  -\frac12\langle a, b\rangle\right|\leq \frac{1}{2}\|a\|\|b\|,
\end{equation}
where $\mathcal{Z}=\{x\},$   $x, a, b\in  \mathcal{X}$ and $\|x\|=1.$ \\

Before we point out some generalization of  \eqref{RichardSelberg}, we collect some results recently obtained by one of the authors in \cite{ACDF, BC}.\\
\begin{lemma}\label{previousresults}
	Let $\mathcal{Z}=\{z_i: i=1, \cdots, n\}$ be a subset of nonzero vectors in $\mathcal{X}$, then \begin{enumerate}
		\item $S_{\mathcal{Z}}$ is a positve operator and $\|S_{\mathcal{Z}}\|\leq 1.$
		\item 	$\|2S_{\mathcal{Z}}-I\|\leq 1.$
		\item For any $a, b \in \mathcal{X}$, we have 
		\begin{eqnarray}\label{ineq}
		\|a\| \|b\|&\geq& |\langle a, b\rangle-\langle S_{\mathcal{Z}}a, b\rangle |+\langle S_{\mathcal{Z}}a, a\rangle^{1/2}\langle S_{\mathcal{Z}}b, b\rangle^{1/2}\nonumber \\
		&\geq& |\langle a, b\rangle|-|\langle S_{\mathcal{Z}}a, b\rangle |+\langle S_{\mathcal{Z}}a, a\rangle^{1/2}\langle S_{\mathcal{Z}}b, b\rangle^{1/2}\nonumber \\
		&\geq& |\langle a, b\rangle|.\nonumber \
		\end{eqnarray}
		\end{enumerate}

\end{lemma}
Now, we generalize  Richard's inequality for any subset $\mathcal{Z}$ contained in $\mathcal{X}$, and we characterize when the equality holds. 
\\
\begin{theorem}\label{generalizacion}
	For any $a, b \in \mathcal{X}$ and $\mathcal{Z}=\{z_i: i=1, \cdots, n\}$ a subset of nonzero vectors in $\mathcal{X}$, it hold 
	\begin{equation}\label{Richardgeneral}
	\left |\langle S_{\mathcal{Z}}a, b\rangle-\frac{1}{2} \langle a, b\rangle\right|\leq \frac12 \|a\| \|b\|.
	\end{equation}
	The case of equality holds in \eqref{Richardgeneral} if and only if
	\begin{equation}
	S_{\mathcal{Z}}a=\frac12 a+\frac12\frac{\|a\|}{\|b\|}e^{i\theta}b,
	\end{equation} 
	for some $\theta\in [0, 2\pi).$
\end{theorem}
\begin{proof}
The first inequality is a simple consequence of Lemma \ref{previousresults}, but to make this article complete, we have included the proof. \\
	Let $a,b \in \mathcal{X},$ then from  Cauchy–
	Schwarz’s inequality and Lemma \ref{previousresults}, we have 
	\begin{eqnarray}
	\left|\langle S_{\mathcal{Z}}a, b\rangle -\frac{1} {2} \langle a, b\rangle \right|&=&\left|\left\langle \left(S_{\mathcal{Z}}-\frac{1} {2} I\right)a, b\right\rangle \right| 
	\leq \frac{1}{2} \left\|2S_{\mathcal{Z}}-I\right\| \|a\| \|b\| \:\:\: \nonumber\\
	&\leq& \frac{1}{2} \|a\| \|b\|.
	\nonumber\
	\end{eqnarray}
	Then, the equality holds, for $y\neq 0$, if and only if
	\begin{equation}\label{equality}
	S_{\mathcal{Z}}a=\frac12 a+\delta b,
	\end{equation}
	for some $\delta \in \mathbb{C}.$ Using the expression \eqref{equality}, we conclude that $\delta=\frac{\|a\|}{\|b\|}e^{i\theta}$
	for some $\theta\in [0, 2\pi).$
	\end{proof}
Incidentally, if $\{a, b\}$ is linearly dependent, then the equality in \eqref{Richardgeneral},
holds for any subset $\mathcal{Z},$ if  and only if $S_{\mathcal{Z}}b=\frac12(1+e^{i\beta})b$
for some $\beta\in [0, 2\pi)$.
\\

\begin{proposition}\label{eortogonal}
	Let $\mathcal{Z}$ be a finite subset of nonzero vectors in $\mathcal{X}.$ If there exists $e\in \mathcal{Z}^{\perp}$ with $\|e\|=1$, then for any $x, y \in \mathcal{X}$ hold 
	\begin{eqnarray}
	\left |\langle S_{\mathcal{Z}}a, b\rangle-\frac{1}{2} \langle a, b\rangle\right|&\leq& \left |\langle S_{\mathcal{Z}}a, b\rangle-\frac{1}{2} \langle a, b\rangle+\frac12 \langle a, e\rangle \langle e, b\rangle \right|+ \frac12\left| \langle a, e\rangle \langle e, b\rangle \right|\nonumber\\
	&\leq& \frac12 \|a\| \|b\|.
	\end{eqnarray}
\end{proposition}
\begin{proof}
	Let $\mathcal{Z}_1=\{e\},$ then  $S_{\mathcal{Z}_1}=e\otimes e$ and by  Lemma \ref{previousresults} we have 
	\begin{equation}\label{refCSDragomir}
	|\langle a,b\rangle|\leq |\langle a,b\rangle-\langle a,e\rangle\langle e,b\rangle|+|\langle a,e\rangle\langle e,b\rangle|\leq \|a\|\|b\|,
	\end{equation}
	for any $a,b \in \mathcal{X}.$ Now, if in the last inequality, which is a refinement of the Cauchy–
	Schwarz’s inequality,  we replace  $a$ by $(S_{\mathcal{Z}}-\frac12I)a$ and we use that $S_{\mathcal{Z}}e=0$, then we obtain

	\begin{eqnarray}
	\left|\left\langle \left(S_{\mathcal{Z}}-\frac12I\right)a,b\right\rangle\right|&\leq& \left|\langle S_{\mathcal{Z}}a, b\rangle-\frac{1}{2} \langle a, b\rangle+\frac12\langle a,e\rangle\left\langle e,b\right\rangle\right|+\frac12\left|\langle a,e\rangle\left\langle e,b\right\rangle\right|\nonumber \\
	&\leq&\frac{1}{2} \left\|2S_{\mathcal{Z}}-I\right\| \|a\| \|b\|\leq\frac{1}{2} \|a\| \|b\| \nonumber \
	\end{eqnarray}
	for any $a, b \in \mathcal{X}.$ 
\end{proof}

\begin{remark}
	Notice that \eqref{refCSDragomir} is also established by Dragomir  in \cite{Dra85}. However, our approach here is different from his.
\end{remark}

We can obtain a  refinement of Richard's inequality, from the previous statement,  by considering the positivity and appropiate set $\mathcal{Z}$.

\begin{corollary}
	For any $x, a, b \in \mathcal{X}$ with $\|x\|=1$, we have
	\begin{align*}
\left\vert \left\langle a,x\right\rangle \left\langle x,b\right\rangle -%
\frac{1}{2}\left\langle a,b\right\rangle
\right\vert  &\leq  \left\vert \left\langle a,x\right\rangle \left\langle x,b\right\rangle -%
\frac{1}{2}\left\langle a,b\right\rangle +\frac12 \left\langle a,e\right\rangle\left\langle e,b\right\rangle
\right\vert+\frac12\left|\langle a,e\rangle\left\langle e,b\right\rangle\right|\nonumber\\
 &\leq \frac{1}{2}\left\Vert
a\right\Vert \cdot \left\Vert b\right\Vert,
	\end{align*}
where $e\in \mathcal{X}$ with $\left\langle x,e\right\rangle=0.$
\end{corollary}

\begin{proof}
		We consider $\mathcal{Z}=\{x\}$, then $e\in \mathcal{Z}^{\perp}$ and by Proposition \ref{eortogonal}, we conclude
			\begin{align*}
				\left\vert \left\langle a,x\right\rangle \left\langle x,b\right\rangle -%
			\frac{1}{2}\left\langle a,b\right\rangle
			\right\vert &=	\left| \langle S_{\mathcal{Z}} a, b\rangle  -\frac12\langle a, b\rangle\right|
		\nonumber\\ &\leq \left\vert \left\langle a,x\right\rangle \left\langle x,b\right\rangle -%
		\frac{1}{2}\left\langle a,b\right\rangle +\frac12 \left\langle a,e\right\rangle\left\langle e,b\right\rangle
		\right\vert+\frac12\left|\langle a,e\rangle\left\langle e,b\right\rangle\right|\nonumber\\
			&\leq \frac{1}{2}\left\Vert
			a\right\Vert \cdot \left\Vert b\right\Vert.
		\end{align*}
		\end{proof}
From Proposition \ref{eortogonal},  we also obtain the following refinement of Buzano type inequality. 
\begin{corollary}
	Let $\mathcal{Z}$ be a finite subset of nonzero vectors in $\mathcal{X}.$ If there exists $e\in \mathcal{Z}^{\perp}$ with $\|e\|=1$, then for any $x, y \in \mathcal{X}$ hold 
	\begin{eqnarray}
	\left |\langle S_{\mathcal{Z}}a, b\rangle\right|&\leq& \left |\langle S_{\mathcal{Z}}a, b\rangle-\frac{1}{2} \langle a, b\rangle+\frac12 \langle a, e\rangle \langle e, b\rangle \right|+\left |\frac12 \langle a, e\rangle \langle e, b\rangle \right|+\frac{1}{2} |\langle a, b\rangle|\nonumber\\
	&\leq& \frac12 \left( |\langle a, b\rangle|+\|a\| \|b\|\right).
	\end{eqnarray}
\end{corollary}

Using the argument of the proof of Theorem \ref{generalizacion}, with different subsets $\mathcal{Z}$ contained in $\mathcal{X}$,  we get the following result
which is the corresponding complex version of Precupanu’s inequality (see  (\ref{7.1})). 

\begin{proposition}\label{prepa}
	Let $a, b, w, z\in \mathcal{X}$ with $w$ and $z$ nonzero vectors. Then 
	\begin{equation}\label{precupanigen}
	\left|\frac{\langle a, w\rangle \langle w, b\rangle}{\|w\|^2}+\frac{\langle a, z\rangle \langle z, b\rangle}{\|z\|^2}-2\frac{\langle a, w\rangle \langle w,z\rangle \langle z, b\rangle}{\|w\|^2\|z\|^2}-\frac 12\langle a, b\rangle \right|\leq \frac12 \|a\|\|b\|.
	\end{equation}
\end{proposition}
\begin{proof}
	We consider the following sets $\mathcal{Z}_1=\{w\}$ and $\mathcal{Z}_2=\{z\}$  contained in $\mathcal{X}.$ Thus, 
	\begin{equation*}
	\frac{\langle a, w\rangle \langle w, b\rangle}{\|w\|^2}+\frac{\langle a, z\rangle \langle z, b\rangle}{\|z\|^2}-2\frac{\langle a, w\rangle \langle w,z\rangle \langle z, b\rangle}{\|w\|^2\|z\|^2}=\langle S_{\mathcal{Z}_1}a, b\rangle+\langle S_{\mathcal{Z}_2}a, b\rangle-2\langle S_{\mathcal{Z}_1}a,  S_{\mathcal{Z}_2}b\rangle.
	\end{equation*}
	
	As $S_{\mathcal{Z}_2}$ is a positive operator and, in particular a selfadjoint operator, we get that 
	\begin{eqnarray}
	\frac{\langle a, w\rangle \langle w, b\rangle}{\|w\|^2}+\frac{\langle a, z\rangle \langle z, b\rangle}{\|z\|^2}-2\frac{\langle a, w\rangle \langle w,z\rangle \langle z, b\rangle}{\|w\|^2\|z\|^2}	&=&\langle (S_{\mathcal{Z}_1}+S_{\mathcal{Z}_2}-2S_{\mathcal{Z}_2}S_{\mathcal{Z}_1}) a, b\rangle. \nonumber \
	\end{eqnarray}
	We remark that 
	$$
	S_{\mathcal{Z}_1}+S_{\mathcal{Z}_2}-2S_{\mathcal{Z}_2}S_{\mathcal{Z}_1}-\frac12 I=(-2)\left(S_{\mathcal{Z}_1}-\frac12 I\right)\left(S_{\mathcal{Z}_2}-\frac12 I\right).
	$$
	Therefore, 
	\begin{eqnarray}
	\left|\bigg\langle \left(S_{\mathcal{Z}_1}+S_{\mathcal{Z}_2}-2S_{\mathcal{Z}_2}S_{\mathcal{Z}_1}-\frac12 I\right) a, b\bigg\rangle\right|&=&\left|\bigg\langle (-2)\left(S_{\mathcal{Z}_1}-\frac12 I\right)\left(S_{\mathcal{Z}_2}-\frac12 I\right) a, b\bigg\rangle\right|\nonumber \\
	&\leq & 2\left\|S_{\mathcal{Z}_1}-\frac12 I\right\| \left\|S_{\mathcal{Z}_2}-\frac12 I\right\| \|a\| \|b\|\nonumber \\
	&\leq& \frac 12 \|a\| \|b\|. \nonumber \
	\end{eqnarray}
\end{proof}

From \eqref{precupanigen}, we get the following generalization of Buzano's inequality. 
\begin{corollary}
	For any $a, b, w, z \in \mathcal{X}$ with $w\neq 0$ and $z\neq 0$, it hold 
	\begin{equation}\label{precupanigen2}
	\left|\frac{\langle a, w\rangle \langle w, b\rangle}{\|w\|^2}+\frac{\langle a, z\rangle \langle z, b\rangle}{\|z\|^2}-2\frac{\langle a, w\rangle \langle w,z\rangle \langle z, b\rangle}{\|w\|^2\|z\|^2}\right|\leq \frac12 (|\langle a, b\rangle|+ \|a\|\|b\|).
	\end{equation}
	
\end{corollary}

In particular, if $\langle z, b\rangle=0$ in \eqref{precupanigen2},  then we obtain the Buzano's inequality. 

Motivated by the proof of Proposition \ref{prepa}, we obtain the following statement.

\begin{theorem}\label{refCScontraction}
	Let $\mathcal{Z}_1, \cdots, \mathcal{Z}_n$ be finite subsets of nonzero vectors in $\mathcal{X}
,$ then for any $a, b \in \mathcal{X}$  and $z_k\in \mathbb{C}$, $k=1,...,n$,  hold 
	\begin{equation}\label{buzanosuma}
	\left |\bigg\langle \sum_{k=1}^n z_k\left(S_{\mathcal{Z}_k}-\frac12I\right)a, b\bigg\rangle\right|\leq \frac{\sum_{k=1}^n |z_k|}{2} \|a\| \|b\|,
	\end{equation}
	and 
	\begin{equation}\label{Precupanuproduct2}
	\left |\bigg\langle \prod_{k=1}^n z_k\left(S_{\mathcal{Z}_k}-\frac12I\right)a, b\bigg\rangle\right|\leq \frac{\prod_{k=1}^n |z_k|}{2^n} \|a\| \|b\|.
	\end{equation}
	
\end{theorem}
\begin{proof}
	It is consequence of the  fact that $\|\cdot\|$ is a submultiplicative norm on $\mathcal{B}(\mathcal(X))$, the triangle inequality  and Lemma \ref{previousresults}. 
\end{proof}

The inequality \eqref{Precupanuproduct2} is a generalization of Precupanu's inequality. 

In particular, if $(-1)^n\prod_{k=1}^n z_k=2^{n-1}$, then 
	\begin{equation}\label{Precupanuproduct}
	\left |\bigg\langle \left(\prod_{k=1}^n z_kS_{\mathcal{Z}_k}\right)a, b\bigg\rangle- \frac12\langle a, b\rangle \right|\leq \frac12 \|a\| \|b\|.
	\end{equation}

\begin{corollary}
	Let $\mathcal{Z}_1, \cdots, \mathcal{Z}_n$ be finite subsets of nonzero vectors in $\mathcal{X}$ and $z_1, \cdots, z_n$ complex numbers such that $\sum_{k=1}^n |z_k|=\sum_{k=1}^n z_k=1.$ Then, for any $a, b \in \mathcal{X}$   hold 
	
	\item \begin{equation}
	\left |\bigg\langle \left(\sum_{k=1}^n z_kS_{\mathcal{Z}_k}\right)a, b\bigg \rangle -\frac12\langle a, b\rangle \right|\leq \frac{1}{2} \|a\| \|b\|.
	\end{equation}
\end{corollary}
\begin{proof}
	By the hypothesis $\sum_{k=1}^n z_k=1$ we get 
	\begin{equation}
	\left(\sum_{k=1}^n z_kS_{\mathcal{Z}_k}\right) -\frac12 I= \sum_{k=1}^n z_k\left(S_{\mathcal{Z}_k}-\frac12I\right).
	\end{equation}
	Then, for any $a, b\in \mathcal{X}$, we have as consequence of \eqref{buzanosuma}
	\begin{eqnarray}
	\left |\bigg\langle \left(\sum_{k=1}^n z_kS_{\mathcal{Z}_k}\right)a, b\bigg \rangle -\frac12\langle a, b\rangle \right|&=& \left |\bigg\langle \sum_{k=1}^n z_k\left(S_{\mathcal{Z}_k}-\frac12I\right)a, b\bigg\rangle\right| \nonumber \\&\leq& \frac{\sum_{k=1}^n |z_k|}{2} \|a\| \|b\|= \frac{1}{2} \|a\| \|b\|.
	\end{eqnarray}
\end{proof}

As consequence of the previous result, we attain  a generalization of Buzano's inequality. 

\begin{proposition}
	Let $\mathcal{Z}_1, \cdots, \mathcal{Z}_n$ be finite subsets of nonzero vectors in $\mathcal{X}$ and $z_1, \cdots, z_n$ complex numbers such that $\sum_{k=1}^n |z_k|=\sum_{k=1}^n z_k=1.$ Then, for any $x, y \in \mathcal{X}$   hold 
	
	\item \begin{equation}
	\left |\bigg\langle \left(\sum_{k=1}^n z_kS_{\mathcal{Z}_k}\right)a, b\bigg \rangle  \right|\leq \frac{1}{2}(|\langle a, b\rangle|+ \|a\| \|b\|).
	\end{equation}
\end{proposition}
\begin{theorem}
In an Hilbert space $\mathcal{X}$ over the
field of complex numbers $\mathbb{C}$, we have
\begin{equation}
\|\alpha S_{\mathcal{Z}}-\beta I\|\leq\max\{|\beta|,|\alpha-\beta|\}  \label{16.3}
\end{equation}
for all $a,b,x\in \mathcal{X}$, $\mathcal{Z}=\{x\}$, $x\neq 0$, and for every $\alpha,\beta \in \mathbb{C}$.
\end{theorem}
\begin{proof}
From inequality \eqref{16} we deduce
\begin{equation}
|\langle \left(\alpha S_{\mathcal{Z}}-\beta I\right)a,b\rangle|\leq\max\{|\beta|,|\alpha-\beta|\}\|a\|\|b\| . \label{16.1}
\end{equation}
Taking  the supremum in above relation for $\|a\|=\|b\|=1$, we find the inequality of the statement.
\end{proof}

\subsection*{Declarations}
\begin{itemize}
\item {\bf{Availability of data and materials}}: Not applicable.
\item {\bf{Competing interests}}: The authors declare that they have no competing interests.
\item {\bf{Funding}}: Not applicable.
\item {\bf{Authors' contributions}}: Authors declare that they have contributed equally to this paper. All authors have read and approved this version.
\end{itemize}
\bibliographystyle{amsplain}

\begin{thebibliography}{99}
\bibitem{1} Aldaz, J. M., Strengthened Cauchy-Schwarz and H\"{o}lder
inequalities, \textit{J. Inequal. Pure Appl. Math.}, \textbf{10} (4), art.
116, 2009.
\bibitem{ACDF}  Altwaijry, N., Conde, C., Dragomir, S. S., Feki, K.,  Some refinements of Selberg Inequality and related results, \textit{Symmetry}, 15 (8), (2023) (doi.org/10.3390/sym15081486).
\bibitem{2} Alzer, H, A Refinement of the Cauchy-Schwarz Inequality, \textit{%
Journal of Mathematical Analysis and Applications}, \textbf{168}, 1992,
596-604.


\bibitem{BC} Bottazzi, T., Conde, C., Generalized Buzano Inequality, \textit{Filomat},  37 (2023), 27.
\bibitem{3} Buzano, M.L., Generalizzazione della disiguaglianza di
Cauchy-Schwarz (Italian), \textit{Rend. Sem. Mat. Univ. e Politech. Torino} 
\textbf{31} (1974) 405--409.

\bibitem{4} Cauchy, A.-L., \textit{Cours d'Analyse de l'\'{E}cole Royale
Polytechnique, I \`{e}re partie, Analyse Alg\'{e}brique}, Paris, 1821.
Reprinted by Ed. Jacques Gabay, Paris, 1989.

\bibitem{Dra85} Dragomir, S. S., Some refinements of Schwartz inequality, \textit{ Simpozionul de Matematici \c si Aplica\c tii},
Timi\c soara, Romania 1–2 (1985), 13–16.

\bibitem{6} Dragomir, S.S., Refinements of Buzano's and Kurepa's
inequalities in inner product spaces, FACTA UNIVERSITATIS (NI\v{S}) \textit{%
	Ser. Math. Inform}. \textbf{20} (2005), 65-73.


\bibitem{5} Dragomir, S.S., A potpourri of Schwarz related inequalities in
inner product spaces (II), \textit{J. Inequal. Pure Appl. Math.} \textbf{7
(1)} (2006), Article 14.

\bibitem{DG} Dragomir, S. S., Go\c sa, A. C., A generalisation of an Ostrowski inequality in inner product spaces, Inequality theory and applications. Vol. 4, 61--64, Nova Sci. Publ., New York, 2007. 

\bibitem{fujiikubo} Fujii, M., Kubo, F., Buzano  inequality and bounds for roots of algebraic equations, \textit{Proc. Amer. Math. Soc.} \textbf{117} (1993),  2, 359--361.

\bibitem{7} Gavrea, I., An extention of Buzano's inequality in inner product
space, \textit{Math. Inequal. Appl.}, 10 (2007), 281-285.

\bibitem{8} Khosravi , M., Drnov\v{s}ek , R., Moslehian, M., S., A
commutator approach to Buzano's inequality, \textit{Filomat} \textbf{26}:4
(2012), 827--832, DOI 10.2298/FIL1204827K.

\bibitem{9} Lupu, C. Schwarz, D., Another look at some new Cauchy-Schwarz
type inner product inequalities, \textit{Applied Mathematics and Computation}%
, \textbf{231} (2014), 463-477.

\bibitem{MPF} Mitrinovi\'{c}, D. S., Pe\v{c}ari\'{c}, J. E., Fink, A. M., 
\textit{Classical and new inequalities in analysis}, Mathematics and its Applications (East European Series), 61, Kluwer Academic Publishers Group, Dordrecht, 1993.

\bibitem{9.1} Minculete, N., Considerations about the several inequalities in an inner product space, \textit{J. Math. Ineq.}, \textbf{12} (1), (2018), 155--161.

\bibitem{11} Ostrowski, A., \textit{Vorlesungen \"{u}ber Differential-und
Integralrechnung}, Vol. \textbf{2}, Birkhauser, Basel, 1951.

\bibitem{12} Pe\v{c}ari\'{c}, J., Raji\'{c}, R., The Dunkl-Williams equality
in pre-Hilbert $\mathbb{C}^{\ast }$-modules. \textit{Lin. Algebra Appl}., 
\textbf{425} (2007). 16-25.

\bibitem{13} Popa, D, Ra\c{s}a, I., Inequalities involving the inner
product, \textit{J. Inequal. Pure Appl. Math.} \textbf{8} (3) (2007),
Article 86, 4 pp.

\bibitem{14} Precupanu, T, On a generalisation of Cauchy-Buniakowski-Schwarz
inequality, \textit{Anal. St. Univ. \textquotedblleft Al. I.
Cuza\textquotedblright\ Ia\c{s}i}, \textbf{22}(2) (1976), 173-175.

\bibitem{15} Richard, U., Sur des in \'{e}galit\'{e}s du type Wirtinger et
leurs application aux \'{e}quationes diffrentielles ordinarires, \textit{Colloquium
of Anaysis held in Rio de Janeiro}, (1972), 233--244.
\end{thebibliography}

\end{document}